\documentclass[11pt,reqno]{amsart}
\usepackage{amsmath,amssymb,rawfonts,enumerate}
\usepackage{tikz}
\usepackage{color}
\usepackage{soul}
\usepackage[english]{babel}
\usepackage{graphicx}
\usepackage[utf8]{inputenc}
\newcommand{\be}{\begin{equation}}
\newcommand{\ee}{\end{equation}}

\newcommand{\N}{\mathbb{N}}
\newcommand{\R}{\mathbb{R}}

\newtheorem{theorem}{Theorem}[section]
\newtheorem{lemma}[theorem]{Lemma}
\newtheorem{defi}[theorem]{Definition}
\newtheorem{prop}[theorem]{Proposition}
\newtheorem{coro}[theorem]{Corollary}

\begin{document}  
\title{PSEUDOCOMPACT AND PRECOMPACT SUBSEMIGROUPS OF TOPOLOGICAL GROUPS}
\author{JULIO CÉSAR HERNÁNDEZ ARZUSA}

\maketitle 

\begin{abstract}	In this paper we give sufficient conditions under which a subsemigroup of a topological group is a subgroup, adding to the results given in \cite{Kosh, can, axioms, forum, Hof, cc, locally} where conditions exist (such as locally compactness, compactness, feeble compactness and sequential compactness)  for a semigroup to be a group. In our work we proved that closed (or open) precompact subsemigroups of topological groups are semigroups, just like open pseudocompact monoids of topological groups.

\end{abstract}




\section{Preliminaries}
We note by $\R$, $\N$ and $\mathbb{T}$ the set of real numbers, the set of  positive integer numbers  and the unit circle $\R/\mathbb{Z}$, respectively.\\\\Given a space $X$ if $A\subseteq X$, then $\overline{A}$ will note the closure of $A$ in $X$.\\\\
A space $X$ is a \emph{pseudocompact space}  if it is Tychonoff and each continuous mapping $f\colon X\longrightarrow \R$ is bounded. It is easy to prove that if $A$ is a pseudocompact subspace of a space $X$, then $\overline{A}$ is pseudocompact. If $X$ is pseudocompact and $f\colon X\longrightarrow Y$ is a continuous function, being $Y$ a Tychonoff space, then $f(X)$ is pseudocompact. \\\\ A countable  intersection of open sets of a space $X$ is said to be a  $G_\delta$ in  $X$.\\$A\subseteq X$ is a \emph{retract}  if there is a continuous mapping $r\colon X\longrightarrow A$ such that $r(x)=x$ for every $x\in A$. It is not hard to prove that if $A$ is a retract of $X$, then each continuous mapping $f\colon A\longrightarrow Y$ can be extended to a continuous mapping $g\colon X\longrightarrow Y$. From this it follows that if $X$ is pseudocompact and $A$ is a retract, then $A$ is pseudocompact.  It is known that an open closed subset of a space $X$ is a retract of $X$.\\\\
We note by $\R$, $\N$ and $\mathbb{T}$ the set of real numbers, the set of  positive integer numbers  and the unit circle $\R/\mathbb{Z}$, respectively.\\\\
A \emph{topological group} is an abstract group $G$ with topology such that multiplication
on $G$, considered as a mapping of $G \times G$ to $G$, and inversion are both continuous. A subset $S\neq \emptyset$ of group $G$, is called a \emph{subsemigroup} if $SS\subseteq S$. If also, $S$ contains the neutral element of $G$, we say that $S$ is a \emph{submomoid} of $G$. If $G$ is a topological group and $S$ is a subsemigroup of $G$, then the continuity of the multiplication of $G$ implies that $\overline{S}$ is a subsemigroup of $G$. If $V$ is a neighborhood of neutral element in a topological group, we say that $V$ is \emph{symmetric} if $V=V^{-1}$. If $G$ is a topological group, then $N_s(e)$ (resp. $N(e)$) will note the symmetric neighborhoods of $e$ (resp. neighborhoods of $e$) , where $e$ is neutral element of $G$. If $S$ is a topological semigroup and $a\in S$, then the function $r_a\colon S\longrightarrow S$ (resp. $l_a\colon S\longrightarrow S$) defined by $r_a(x)=xa$ (resp. $l_a(x)=ax$), for each $x\in S$, is called the \emph{right shift} (resp. \emph{left shift}) at $a$. Obviously, if $S$ is a topological semigroup $l_a$ and $r_a$ are continuous for every $a\in S$. We say that a topological semigroup $S$ has open shifts, if $r_a$ and $l_a$ are open for each $a\in S$.\\\\
The literature exhibits a variety of sufficient conditions for  a cancellative topological semigroup $S$ to be a topological group, Indeed this is true if
$S$ is
\begin{enumerate}[i)]
\item  compact (\cite[Proposition A4.34]{Hof}),
\item countably compact first countable (\cite[Corollary 5]{cc})
\item sequentially compact (\cite[Theorem 6]{can}),
\item commutative locally compact pseudocompact  with open shifts (\cite[Corollary 10]{forum}),
\item commutative and  locally compact connected with open shifts (\cite[Theorem 4]{axioms}), or
\item commutative feebly compact first countable regular with open shifts 
(\cite[Theorem 3]{axioms}). 
\end{enumerate}

We now increase this list of results and include the case in which $S$ is an open pseudocompact submoniod or a open subsemigroup of a topological group.

\section{Main Results}

Every space  considered in this section is assumed to be $T_2$.
\noindent
\begin{defi}\label{111}
A subset $A$ of a topological  group  $G$  is precompact  in $G$, if given $V\in N(e)$ there exists a finite set $K\subseteq G$ such that $A\subseteq KV$ and $A\subseteq VK$.
\end{defi}

From Definition \ref{111} we obtain that compact subsets of a topological group are precompact. \cite[Theorem 3.7.2]{Ar} establishes  that a pseudocompact topological group is precompact. The following theorem extends this result to a class of to\-po\-lo\-gi\-cal semigroups.

\begin{lemma}\label{810}Let $G$ be  a topological group. Then each open pseudocompact  submonoid of $G$ is precompact.
\end{lemma}

\begin{proof}Let $S$ a pseudocompact open monoid of $G$, then $H=S\cap S^{-1}$ is an open subgroup of $S$. Since $S=\bigcup_{s\in S} sH$, the pseudocompactness of $S$ implies that there exist $s_1,s_2,...,s_n \in S$ such that $\{s_1H, s_2H,..., s_nH\}$ is a partition of $S$, formed by open closed subsets of $S$. Thus each $s_iH$,  is a retract of $S$, then for $i=1,2,...,n$, we have that $s_iH$ is pseudocompact. Since the continuous images of a pseudocompact space is pseudocompact, we have that  $H=l_{s_1^{-1}}(s_1H)$ is pseudocompact too. 	From  \cite[Theorem 3.7.2]{Ar} we have that $H$ is precompact. Since the unipunctual sets are precompact, from \cite[Corollary 3.7.11]{Ar} we have that $s_iH$ is precompact  for each $i=1,2,...,n$. Thus $S$ is precompact for being finite union of precompact sets.
\end{proof}

\begin{lemma}\label{1810}(\cite[Lemma 3.7.26]{Ar})Let $A$ be a precompact subset of a topological group  $G$. If $\{x_n:n\in \omega\}\subseteq A$, then $e_G\in \overline{\{x_ix_j^{-1}:i<j, i,j\in \omega\}}$.
\end{lemma}

Compact subsemigroups of  topological groups are subgroups, the following proposition tell us that compactness can be weakened to precompactness.

\begin{prop}\label{8102}Let $S$ be a precompact subsemigroup of topological group $G$, then $\overline{S}$ is  a precompact subgroup of $G$.
\end{prop}

\begin{proof}Let $x\in \overline{S}$. By \cite[Lemma 3.7.5]{Ar}  $\overline{S}$ is a precompact subsemigroup. Note that $\{x^{2}, x^{3},...\}\subseteq \overline{S}$. Then Lemma \ref{1810} implies that $e_G\in \overline{\{x^{-1}, x^{-2},...\}}$. Since the map $a\mapsto a^{-1}\colon G \to G$ is a homeomorphism, we have that $e_G=e_G^{-1}\in (\overline{\{x^{-1}, x^{-2},...\}})^{-1}=\overline{\{x^{1}, x^{2},...\}}\subseteq \overline{S}$, hence $\overline{\{e, x^{1}, x^{2},...\}}\subseteq \overline{S}$. Since the map $a\mapsto x^{-1}a\colon G\to G$,  is a homeomorphism, we have that $x^{-1}=l_{x^{-1}}(e_G)\in l_{x^{-1}}(\overline{\{x^{1}, x^{2},...\}})=\overline{\{e, x^{1}, x^{2},...\}}\subseteq \overline{S}$. The Lemma is proven
\end{proof}

\noindent\cite[Corollary 3.4]{tomita} establishes that  the closure of every subsemigroup of a precompact group is a subgroup. Since subsets of precompact sets are precompact, then this result is a consequence of Proposition \ref{8102}.\\\\
\noindent A subgroup $N$ of a group $G$ is called \emph{normal} or \emph{invariant} if $gNg^{-1}=N$, for each $g\in G$. Also, if $G$ is a topological group, a subgroup $H$ of $G$ is called \emph{admissible} if $H=\bigcap_{n\in \N} U_n$, where $U_n\in N_s(e)$ and $U_{n+1}^{2}\subseteq U_n$ for each $n\in N$. It was proven in \cite{pseudo} that any admissible subgroup of a topological group $G$ is closed in $G$.

\begin{theorem}\label{2810}Let $S$ be a precompact pseudocompact subsemigroup of a topological group $G$ containing an admissible normal subgroup of $G$. Then $S$ is a subgroup of $G$.
\end{theorem}

\begin{proof}Let $N$ be an admissible subgroup of $G$ such that $N\subseteq S$. Now,  $\overline{S}$ is a precompact and pseuducompact subgroup by Proposition \ref{8102}, since $N$ is closed in $G$, then  \cite[Corollary 2.3.13]{pseudo} implies that $\overline{S}/N$ is a compact subgroup of the $T_2$ group  with a countable base $G/N$, particularly $\overline{S}/N$ is metrizable. If $\pi \colon \overline{S}\longrightarrow \overline{S}/N$ is the respective quotient mapping, then $\pi(S)$ is a pseudocompact and metrizable subsemigroup of $ \overline{S}/N$. From \cite[Corollary 1.1.13]{pseudo} we have that $\pi(S)$ is a compact subsemigroup of $\overline{S}/N$, hence $\pi(S)$ is a subgroup according to \cite[Proposition A4.34]{Hof}. Let us see $S$ is a subgroup. Indeed let $x\in S$, then $x^{-1}N=(xN)^{-1}\in S/N$, this means that there exists  $s\in S$ such that $x^{-1}N=sN$, therefore $x^{-1}\in sN\subseteq sS\subseteq S$, this proves that $S$ is a subgroup of $\overline{S}$.
\end{proof}

\begin{prop}\label{910}
Let $G$ be a topological group and let $P$ be   $G_\delta$ in $G$ such that $e_G\in P$. Then any precompact pseudocompact subsemigroup of $G$ containing $P$ is a subgroup of $G$.
\end{prop}

\begin{proof}If $S$ is a subsemigroup such that $P\subseteq S\subseteq \overline S$, then $P$ is  $G_\delta$ in  $\overline{S}$, which is a group by Proposition \ref{8102}. From \cite[Lemma 2.3.9]{pseudo} $P$ contains a normal admissible subgroup of $\overline{S}$. Then the Theorem \ref{2810} implies that $S$ is a subgroup of $G$.
\end{proof}

\begin{theorem}\label{2910}Each open pseudocompact  submonid of a topological group $G$ is a precompact closed subgroup of $G$.
\end{theorem}

\begin{proof}
Let $S$ be an open pseudocompact monoid of a topological group $G$. From Lemma \ref{810}  and Proposition \ref{8102}  we have that $S$ is a precompact pseudocompact open submonoid of the group $\overline{S}$. The Proposition \ref{910} implies that $S$ is a subgroup of $\overline{S}$, since open subgroups  of topological groups are closed, then $S$ is closed.
\end{proof}

Let $S$ be a topological semigroup. We say that $S$ has \emph{continuous division} if give $x,y\in S$ and an open set $V$ in $S$, containing $y$, there are open sets in $S$, $U$, and $W$, containing $x$ and $xy$ ($yx$), respectively; such that $W\subseteq \bigcap\{uV:u\in U\}$ ($W\subseteq \bigcap\{Vu:u\in U\}$). From \cite[Theorem 1.15 and Theorem 1.19]{hom} we know that if $S$ is a commutative cancellative topological semigroup with open shifts and continuous division, then $S$ can be embedded in a topological group $G$. Therefore, if additionally  $S$ is pseudocompact, then the Theorem \ref{2910} tell us that $S$ is a group.

\begin{coro}\label{1010}Let $G$ be a locally compact topological group and let $S$ be an open pseudocompact  submonoid of $G$. Then $S$ is a compact subgroup of $G$.
\end{coro}

\begin{proof}
Since $Cl_G(S)\cap S=S$,  from Corollary 3.3.10 of \cite{En} we have that $S$ is locally compact. Moreover, from Theorem \ref{2910} $S$ is a pseudompact subgroup of $G$. But from Theorem 2.3.2 of \cite{pseudo} we have that each locally compact and pseudocompact group is  compact. Therefore $S$ is a compact subgroup.
\end{proof}

\begin{lemma}(\cite[Lemma 3.7.7]{Ar})Let $B$ be a precompact subset of a topological group $G$. Then, for every neighbourhood $U$ of the identity $e$ in $G$, there exists a neighbourhood $V$ of $e$ in $G$ such that $bVb^{-1}\subseteq U$, for each $b\in  B$.
\end{lemma}

In the following theorem we extend the Theorem \ref{2910} and we used precompactness instead of pseudocompactness, but we have weakened to be submonoid to subsemigroup.

\begin{theorem}\label{2012}
Each open precompact subsemigroup of a topological group $G$ is a closed subgroup of $G$.
\end{theorem}

\begin{proof}
Let $S$ be an open precompact subsemigroup of $G$. From Proposition \ref{8102} we have that $\overline{S}$ is a precompact subgroup of $G$. 	Therefore $K=S\cap S^{-1}$ is an open subgroup of $\overline{S}$. Let $H=\bigcap_{s\in \overline{S}} sKs^{-1}$, then $H$ is a normal subgroup of $G$, we will see that $H$ is open in $\overline{S}$. Indeed, since $K\in N(e)$ and $\overline{S}$ is precompact, there is $V\in N(e)$ such that $s^{-1}Vs\subseteq K$ for every $s\in \overline{S}$, that is, $V\subseteq \bigcap_{s\in \overline{S}} sKs^{-1}=H$, therefore $H$ is open. Let us consider the discrete quotient group $\overline{S}/H$ and the respective quotient mapping $\pi \colon \overline{S}\longrightarrow \overline{S}/H$. Since $\overline{S}$ is precompact and $\pi$ is a continuous homomorphism, then $\overline{S}/H$ is precompact and discrete and therefore it is finite. Thus $\pi(S)$ is a subgroup of $G$. Let $s\in S$, then $\pi(s)=sH\in \pi(S)$, since $\pi(S)$ is a group, hence $(sH)^{-1}=s^{-1}H\in \pi(S)$, so that there exists $x\in S$ such that $s^{-1}H=xH$. Thus $x^{-1}s^{-1}\in H\subseteq x^{-1}Kx\subseteq x^{-1}Sx$. We conclude  that $s^{-1}\in Sx\subseteq S$, this proves that $S$ is an open  subgroup of $G$. Since that open subgroups are closed, this finishes the proof.
\end{proof}

The Theorem \ref{2910} has been seen as a consequence of Lemma \ref{810} and  the Proposition \ref{910}, however it is also an immediate consequence of the Theorem \ref{2012}. The openness  in the Theorem \ref{2012} cannot be removed. For example,  $S=(\mathbb{Z}+ \sqrt{2}\mathbb{N})/\mathbb{Z}$ is a precompact subsemigroup that fails to be a subgroup.\\\\
 We say that a space $X$ has \emph{countable cellularity } or has \emph{the Souslin property} if every family of pairwise disjoint open sets in $X$ has cardinality less than or equal to $\aleph_0$. From \cite{pseudo} we know that the Souslin property Property is in\-he\-ri\-ted to dense subspaces. The class of the spaces having the Souslin property is large. For example, contains the $\sigma$-compact paratopological groups(see \cite[Corollary 2.3]{T3} and the subsemigroups of precompact groups (see \cite[Corollary 3.6]{tomita}. Now, since each precompact group (and hence each pseudocompact group)  is a dense subgroup of a compact group, then precompact groups have the Souslin property. Therefore Theorem \ref{2910} and Proposition \ref{8102} imply the next result.

\begin{coro}\label{1212}Every precompact subsemigroup of a topological group has countable cellularity. Particularly, each open pseudocomapct submonoid of a topological group has countable cellularity.
\end{coro}

A space $X$ is called \emph{countably compact} if any open cover of $X$ has a numerable subcover. Since each Tychonoff countably compact space is pseudocompact (see \cite[Theorem 3.10.20]{En}),  then 	Corollary \ref{1212} is a partial answer to the Problem 2.5.2 of \cite{Ar}.

\end{document}